\documentclass{article}
\usepackage[margin=0.5in]{geometry}
\usepackage[utf8]{inputenc}
\usepackage{fullpage}
\usepackage{enumerate}
\usepackage{authblk}
\usepackage[colorinlistoftodos]{todonotes}
\usepackage{verbatim}
\usepackage{section, amsthm, textcase, setspace, amssymb, lineno, amsmath, amssymb, amsfonts, latexsym, fancyhdr, longtable, ulem}
\usepackage{epsfig, graphicx, pstricks,pst-grad,pst-text,tikz, colortbl}
\usepackage{epsf}
\usepackage{graphicx, color}
\usepackage{float}
\usepackage[rflt]{floatflt}
\usepackage{amsfonts}
\usepackage{latexsym}
\usetikzlibrary{fit,matrix,positioning}
\usepackage{pdflscape}
\usepackage{kbordermatrix, mathrsfs}
\usetikzlibrary{decorations.pathreplacing}
\usetikzlibrary{decorations.markings}

\newtheorem{lemma}{Lemma}
\newtheorem{theorem}{Theorem}
\newtheorem{proposition}{Proposition}
\newtheorem{definition}{Definition}
\newtheorem{Ex}{Example}
\newtheorem{remark}{Remark}

\newtheorem{conj}{Conjecture}
\newtheorem*{theorem*}{Theorem}

\definecolor{dark green}{rgb}{0.0, 0.5, 0.0}

\newtheorem{innercustomgeneric}{\customgenericname}
\providecommand{\customgenericname}{}
\newcommand{\newcustomtheorem}[2]{%
  \newenvironment{#1}[1]
  {%
   \renewcommand\customgenericname{#2}%
   \renewcommand\theinnercustomgeneric{##1}%
   \innercustomgeneric
  }
  {\endinnercustomgeneric}
}

\newcustomtheorem{customthm}{Theorem}

\newcommand\addvmargin[1]{
  \node[fit=(current bounding box),inner ysep=#1,inner xsep=0]{};}

\begin{document}

\title{The breadth of Lie poset algebras}

\author[*]{Alex Cameron}
\author[**]{Vincent E. Coll, Jr.}
\author[**]{Nicholas Mayers}
\author[**]{Nicholas Russoniello}

\affil[*]{Department of Mathematics, Muhlenberg College, Allentown, PA, 18104}
\affil[**]{Department of Mathematics, Lehigh University, Bethlehem, PA, 18015}

\maketitle

\bigskip
\begin{abstract} 
\noindent
The breadth of a Lie algebra $L$ is defined to be the maximal dimension of the image of $ad_x=[x,-]:L\to L$, for $x\in L$. Here, we initiate an investigation into the breadth of three families of Lie algebras defined by posets and provide combinatorial breadth formulas for members of each family.
\end{abstract}

\section{Introduction}

\textit{Convention:}  We assume throughout that all Lie algebras are over an algebraically closed field of characteristic zero, \textbf{k}, which we may take, without any loss of generality, to be the complex numbers.
\medskip


The complete classification of simple Lie algebras, elegantly couched in the language of root systems and Dynkin diagrams, stands in sharp contrast to the solvable case where the Lie algebras are classified only up to dimension six (see \textbf{\cite{Graaf}}). Recent efforts have concentrated on the seemingly more tractable nilpotent case, where the study of nilpotent Lie algebras and their invariants is of topical interest.

One such invariant is the ``breadth" of a Lie algebra $L$, which is defined to be the maximal dimension of the image of $ad_x=[x,-]:L\rightarrow L$ as $x$ runs through the elements of $L$. This Lie-algebraic invariant was introduced by Leedham-Green, Neumann and Weigold (\textbf{\cite{Green}}, 1969). Some time later, Khuhirun et al. (\textbf{\cite{breadth1}}, 2014) characterized nilpotent Lie algebras of breadth one and two and  provided a full classification in breadth one.  In breadth two, they achieved a succinct classification only up to dimension six. This work was inspired by the recent interest in nilpotent Lie algebras and  analogous breadth work of Parmeggani and Stellmacher (\textbf{\cite{Parm}}, 1999) who gave a characterization of finite p-groups of breadth one and two (the breadth of a finite group is the cardinality of its largest conjugacy class). As a capstone to the work in \textbf{\cite{breadth1}}, Remm (\textbf{\cite{Remm}}, 2017) used characteristic sequences to complete, in particular, the classification of nilpotent Lie algebras of breadth two. Here, rather than studying Lie algebras with a particular, fixed breadth value, we initiate an investigation into explicit breadth formulas for three families of Lie algebras defined by posets: Lie poset algebras, type-A Lie poset algebras, and nilpotent Lie poset algebras. The last two are subalgebras of the first, but only the last of these is nilpotent.

Lie poset algebras are solvable subalgebras of  $\mathfrak{gl}(n)$ which arise naturally from the incidence algebras of posets (cf. \textbf{\cite{Ro}}) and can be defined as follows. For each poset $(\mathcal{P},\preceq_{\mathcal{P}})$ with $\mathcal{P}=\{1,\hdots,n\}$, one obtains a Lie algebra $\mathfrak{g}(\mathcal{P})$ consisting of $|\mathcal{P}|\times|\mathcal{P}|$ matrices whose $i,j$-entry can be nonzero if and only if $i\preceq_{\mathcal{P}} j,$ and the Lie bracket of $\mathfrak{g}(\mathcal{P})$ is given by $[X,Y]=XY-YX$, where juxtaposition denotes standard matrix multiplication. The imposition of various algebraic conditions on the members of $\mathfrak{g}(\mathcal{P})$ yields Lie poset algebras of classical type. For example, if a vanishing trace condition is applied, one obtains a type-A Lie poset algebra, which we denote by $\mathfrak{g}_{A} (\mathcal{P})$. Removing diagonal elements from $\mathfrak{g}(\mathcal{P})$ results in a nilpotent subalgebra of $\mathfrak{gl}(n)$ which, following \textbf{\cite{nilco}}, is denoted by $\mathfrak{g}^{\prec}(\mathcal{P})$ and referred to as a \textit{nilpotent Lie poset algebra.} See \textbf{\cite{CM, nil, CG, nilco}}.

The organization of the paper is as follows. After covering some preliminaries in Section~\ref{sec:prelim}, our main results are detailed in Section~\ref{sec:results}. Theorem~\ref{thm:first} establishes that the breadth of $\mathfrak{g}(\mathcal{P})$ and $\mathfrak{g}_{A}(\mathcal{P})$ is given by the number of relations in $\mathcal{P}$.  In Theorem~\ref{thm:breadthncov}, we consider three families of nilpotent Lie poset algebras (corresponding to chains, trees, and grids) to find that the breadth of a member algebra is simply the number of non-covering relations in the algebra's defining poset. In Theorems 3--6, we examine a three-parameter family of nilpotent Lie poset algebras, $\mathfrak{g}^{\prec}(\mathcal{P}(r_0,r_1,r_2))$, whose underlying posets have Hasse diagrams (see Section~\ref{sec:prelim}) which can be described as ``expanded double-fan'' graphs (see Figure~\ref{fig:Pnrm}). In contrast to previous examples, we show  that the breadth of such a nilpotent Lie poset algebra is generally not a function of the number and type of poset relations, but rather is an elementary function of its parameters. Motivated by this discovery, we end Section~\ref{sec:results} with a combinatorial obstruction to $\mathfrak{g}^{\prec}(\mathcal{P})$ having breadth equal to the number of non-covering relations in $\mathcal{P}$ (see Theorem~\ref{thm:subpo}). Finally, in the epilogue, we discuss directions for further research along with some consequences of the aforementioned obstruction result.

\begin{figure}[H]
$$\begin{tikzpicture}[scale=1.2]
\def\Node{\node [circle, fill, inner sep=1.5pt]}
\Node (1) at (0,0){};
\Node (2) at (1,0){};
\Node (3) at (2,0){};
\Node (4) at (3,0){};
\Node (5) at (1,1){};
\Node (6) at (2,1){};
\Node (7) at (0.5,2){};
\Node (8) at (1.5,2){};
\Node (9) at (2.5,2){};

\draw (1)--(5)--(7);
\draw (1)--(6)--(7);
\draw (2)--(5)--(8);
\draw (2)--(6)--(8);
\draw (3)--(5);
\draw (3)--(6);
\draw (4)--(5);
\draw (4)--(6);
\draw (5)--(9);
\draw (6)--(9);
\end{tikzpicture}$$
    \caption{Hasse diagram of the poset $\mathcal{P}(4,2,3)$}
    \label{fig:Pnrm}
\end{figure}
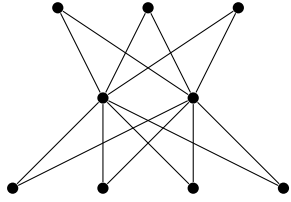

\section{Preliminaries}\label{sec:prelim}

A \textit{finite poset} $(\mathcal{P}, \preceq_{\mathcal{P}})$ consists of a finite set $\mathcal{P}=\{1,\hdots,n\}$ together with a binary relation $\preceq_{\mathcal{P}}$ on $\mathcal{P}$ which is reflexive, anti-symmetric, and transitive. We tacitly assume that if $x\preceq_{\mathcal{P}}y$ for $x,y\in\mathcal{P}$, then $x\le y$, where $\le$ denotes the natural ordering on $\mathbb{Z}$. When no confusion will arise, we simply denote a poset $(\mathcal{P}, \preceq_{\mathcal{P}})$ by $\mathcal{P}$, and $\preceq_{\mathcal{P}}$ by $\preceq$. 

Let $x,y\in\mathcal{P}$. If $x\preceq y$ and $x\neq y$, then we call $x\preceq y$ a \textit{strict relation} and write $x\prec y$. Let $Rel(\mathcal{P})$ denote the set of strict relations between elements of $\mathcal{P}$, $Ext(\mathcal{P})$ denote the set of minimal and maximal elements of $\mathcal{P}$, and $Rel_E(\mathcal{P})$ denote the set of strict relations between the elements of $Ext(\mathcal{P})$.

\begin{Ex}\label{ex:stargate}
Consider the poset $\mathcal{P}=\{1,2,3,4\}$ with $1\prec2\prec3,4$.  We then have $$Rel(\mathcal{P})=\{1\prec 2,1\prec 3,1\prec 4,2\prec 3,2\prec 4\},$$  $$Ext(\mathcal{P})=\{1,3,4\},\quad \text{and}\quad Rel_E(\mathcal{P})=\{1\prec 3, 1\prec 4\}.$$
\end{Ex}

\noindent
For $x,y\in\mathcal{P}$ satisfying $x\prec y$, we set $$[x,y]_{\mathcal{P}}=\{p\in\mathcal{P}~|~x\preceq p\preceq y\}.$$ Recall that, if $x\prec y$ and there exists no $z\in \mathcal{P}$ satisfying $x\prec z\prec y$, then $y$ \textit{covers} $x$ and $x\prec y$ is a \textit{covering relation}.  Using this language, the \textit{Hasse diagram} of a poset $\mathcal{P}$ can be reckoned as the graph whose vertices correspond to elements of $\mathcal{P}$ and whose edges correspond to covering relations.

\begin{Ex}
Let $\mathcal{P}$ be the poset of Example~\ref{ex:stargate}. The Hasse diagram of $\mathcal{P}$ is given in Figure~\ref{fig:Hasse} below.
\begin{figure}[H]
$$\begin{tikzpicture}
	\node (1) at (0, 0) [circle, draw = black, fill = black, inner sep = 0.5mm, label=left:{1}]{};
	\node (2) at (0, 1)[circle, draw = black, fill = black, inner sep = 0.5mm, label=left:{2}] {};
	\node (3) at (-0.5, 2) [circle, draw = black, fill = black, inner sep = 0.5mm, label=left:{3}] {};
	\node (4) at (0.5, 2) [circle, draw = black, fill = black, inner sep = 0.5mm, label=right:{4}] {};
    \draw (1)--(2);
    \draw (2)--(3);
    \draw (2)--(4);
    \addvmargin{1mm}
\end{tikzpicture}$$
\caption{Hasse diagram of $\mathcal{P}$, for $\mathcal{P}=\{1,2,3,4\}$ with $1\prec 2\prec 3,4$}\label{fig:Hasse}
\end{figure}
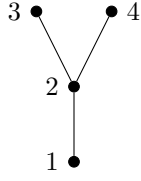
\end{Ex}

\noindent
Ongoing, the collection of non-covering relations of a poset $\mathcal{P}$ will prove important. So, let $$Rel_{\overline{C}}(\mathcal{P})=\{p\prec q~|~p\prec q\text{ is not a covering relation of }\mathcal{P}\}$$ and for $p\in\mathcal{P}$, let $$Rel_{\overline{C}}(\mathcal{P},p)=\{q\prec r\in Rel_{\overline{C}}(\mathcal{P})~|~q=p\text{ or }r=p\}.$$ Given a subset $S\subset\mathcal{P}$, the \textit{induced subposet generated by $S$} is the poset $\mathcal{P}_S$ on $S$, where $i\prec_{\mathcal{P}_S}j$ if and only if $i\prec_{\mathcal{P}}j$.

The following families of posets will be of interest in the sections that follow.

\begin{definition}
Let $C_n$ denote the chain on $n$ elements; that is, the poset on the set $\{1,\hdots,n\}$, where  $$1\prec 2\prec\hdots\prec n.$$
\end{definition}

\begin{remark}
For the families of posets defined in Definitions~\ref{def:grid},~\ref{def:tree}, and~\ref{def:comppos}, it is possible to give definitions in which the underlying set of the posets consist of integers $\{1,\hdots,n\}$, but it is more convenient to use an alternative labeling of the elements.
\end{remark}

\begin{definition}\label{def:grid}
Let $\mathbf{m\times n}$ denote the poset on the set $\{1_1,\hdots,n_1,\hdots,1_m,\hdots,n_m\}$ where $i_j\prec i_{j+1},(i+1)_j$, for $1\leq i\leq n-1$ and $1\leq j\leq m-1$, $n_j\prec n_{j+1}$ , for $1\le j\le m-1$, and $i_m\prec (i+1)_m$, for $1\le i\le n-1$.
\end{definition}

\begin{figure}[H]
$$\begin{tikzpicture}
	\node (1) at (0, 0) [circle, draw = black, fill = black, inner sep = 0.5mm, label=below:{$1_1$}]{};
	\node (2) at (-0.5, 0.5)[circle, draw = black, fill = black, inner sep = 0.5mm, label=left:{$1_2$}] {};
	\node (3) at (0.5, 0.5) [circle, draw = black, fill = black, inner sep = 0.5mm, label=right:{$2_1$}] {};
	\node (4) at (0, 1) [circle, draw = black, fill = black, inner sep = 0.5mm, label=left:{$2_2$}] {};
	\node (5) at (1, 1) [circle, draw = black, fill = black, inner sep = 0.5mm, label=right:{$3_1$}] {};
	\node (6) at (0.5, 1.5) [circle, draw = black, fill = black, inner sep = 0.5mm, label=left:{$3_2$}] {};
	\node (7) at (1.5, 1.5) [circle, draw = black, fill = black, inner sep = 0.5mm, label=right:{$4_1$}] {};
	\node (8) at (1, 2) [circle, draw = black, fill = black, inner sep = 0.5mm, label=above:{$4_2$}] {};
    \draw (1)--(2)--(4);
    \draw (1)--(3)--(4);
    \draw (4)--(6)--(8);
    \draw (3)--(5)--(6);
    \draw (5)--(7)--(8);
\end{tikzpicture}$$
\caption{Hasse diagram of $\mathbf{2\times 4}$}
\end{figure}
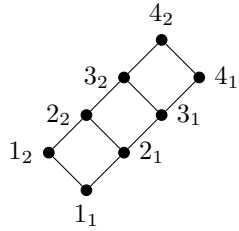

\begin{definition}\label{def:tree}
For $m>1$, let $T_m(n)$ denote the poset on $\{1_1,1_2,2_2,\hdots,m_2,\hdots,1_n,\hdots,m^{n-1}_n\}$ where \\* $i_k\prec (mi-j)_{k+1}$, for $1\le k<n$, $1\le i\le m^{k-1}$, and  $0\le j<m$.
\end{definition}

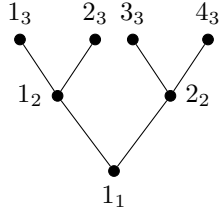
\begin{figure}[H]
$$\begin{tikzpicture}
	\node (1) at (0, 0) {};
	\node (2) at (0, 0.25)[circle, draw = black, fill = black, inner sep = 0.5mm, label=below:{$1_1$}] {};
	\node (3) at (-0.75, 1.25)[circle, draw = black, fill = black, inner sep = 0.5mm, label=left:{$1_2$}] {};
	\node (4) at (0.75, 1.25)[circle, draw = black, fill = black, inner sep = 0.5mm, label=right:{$2_2$}] {};
	\node (5) at (-1.25, 2)[circle, draw = black, fill = black, inner sep = 0.5mm, label=above:{$1_3$}] {};
	\node (6) at (-0.25, 2)[circle, draw = black, fill = black, inner sep = 0.5mm, label=above:{$2_3$}] {};
	\node (7) at (0.25, 2)[circle, draw = black, fill = black, inner sep = 0.5mm, label=above:{$3_3$}] {};
	\node (8) at (1.25, 2)[circle, draw = black, fill = black, inner sep = 0.5mm, label=above:{$4_3$}] {};
	\node (9) at (1, 2) {};
	\draw (3)--(2)--(4);
	\draw (5)--(3)--(6);
	\draw (8)--(4)--(7);
\end{tikzpicture}$$
\caption{Hasse diagram of $T_2(3)$}
\end{figure}

\begin{definition}\label{def:comppos}
Let $\mathcal{P}(r_0, r_1, r_2)$ denote the poset on $\{b_1,\hdots,b_{r_0},m_1,\hdots,m_{r_1},t_1,\hdots,t_{r_2}\}$ with $$b_1,\hdots,b_{r_0}\prec m_1,\hdots,m_{r_1}\prec t_1,\hdots, t_{r_2}.$$
\end{definition}

\begin{Ex}
Using the notation of Definition~\ref{def:comppos}, the poset of Example~\ref{ex:stargate} is $\mathcal{P}(1,1,2)$.
\end{Ex}

Let $\mathcal{P}$ be a finite poset. The (associative) \textit{incidence algebra} $A(\mathcal{P})=A(\mathcal{P}, \textbf{k})$ is the span over $\textbf{k}$ of elements $e_{i,j}$, for $i,j\in\mathcal{P}$ satisfying $i\preceq j$, with product given by setting $e_{i,j}e_{k,l}=e_{i,l}$ if $j=k$ and $0$ otherwise. The \textit{trace} of an element $\sum c_{i,j}e_{i,j}$ is $\sum c_{i,i}.$

We can equip $A(\mathcal{P})$ with the commutator product $[a,b]=ab-ba$, where juxtaposition denotes the product in $A(\mathcal{P})$, to produce the \textit{Lie poset algebra} $\mathfrak{g}(\mathcal{P})=\mathfrak{g}(\mathcal{P}, \textbf{k})$. If $|\mathcal{P}|=n$, then both $A(\mathcal{P})$ and $\mathfrak{g}(\mathcal{P})$ may be regarded as subalgebras of the algebra of $n \times n$ upper-triangular matrices over $\textbf{k}$. Such a matrix representation is realized by replacing each basis element $e_{i,j}$ by the $n\times n$ matrix $E_{i,j}$ containing a 1 in the $i,j$-entry and 0's elsewhere. The (associative) product between elements $e_{i,j}$ is then replaced by matrix multiplication between the $E_{i,j}$.

\begin{Ex}\label{ex:posetmat}
Let $\mathcal{P}$ be the poset of Example~\ref{ex:stargate}. The matrix form of elements in $\mathfrak{g}(\mathcal{P})$ is illustrated in Figure~\ref{fig:tA}, where the $*$'s denote potentially non-zero entries from \textup{\textbf{k}}.
\begin{figure}[H]
$$\kbordermatrix{
    & 1 & 2 & 3 & 4  \\
   1 & * & * & * & *   \\
   2 & 0 & * & * & *  \\
   3 & 0 & 0 & * & 0  \\
   4 & 0 & 0 & 0 & *  \\
  }$$
\caption{Matrix form defining $\mathfrak{g}(\mathcal{P})$, for $\mathcal{P}=\{1,2,3,4\}$ with $1\prec 2\prec 3,4$}\label{fig:tA}
\end{figure}
\end{Ex}

\noindent
Restricting $\mathfrak{g}(\mathcal{P})$ to trace-zero matrices yields a subalgebra of the first classical family $A_{n-1}=\mathfrak{sl}(n)$ which we refer to as a \textit{type-A Lie poset algebra} and denote by $\mathfrak{g}_A(\mathcal{P})$. Restricting $\mathfrak{g}(\mathcal{P})$ to strictly upper-triangular matrices yields a subalgebra which we refer to as a \textit{nilpotent Lie poset algebra} and denote by $\mathfrak{g}^{\prec}(\mathcal{P})$. 

\section{Results}\label{sec:results}
We begin this section with the formal definition of the breadth of a Lie algebra.

\begin{definition}
The \text{breadth} of a Lie algebra $L$ is the invariant
$$
b(L)=\max_{x\in L}~\textup{rank}(ad_x).
$$
\end{definition}

The following result describes bounds on the breadth of a Lie algebra and will be useful in what follows -- the proofs can be found in \textbf{\cite{dissertation}}.

\begin{proposition}\label{prop:ub} If $L$ is a Lie algebra with center $Z(L)$ and derived algebra $[L,L]$, then
\begin{enumerate}[\textup(i\textup)] 
    \item $b(L)\le \dim([L,L])$, and
    \item $b(L)\le \dim(L/ Z(L))-1$.
\end{enumerate}
\end{proposition}

\subsection{(Type-A) Lie poset algebras}

To determine a combinatorial translation of Proposition~\ref{prop:ub}(\textit{i}) for (type-A) Lie poset algebras, we make use of the following Proposition.

\begin{proposition}\label{prop:derivlp}
If $\mathcal{P}$ is a poset and $L=\mathfrak{g}(\mathcal{P})$ or $\mathfrak{g}_A(\mathcal{P})$, then $\dim([L,L])=|Rel(\mathcal{P})|$.
\end{proposition}
\begin{proof}
We claim that $$[L,L]=\text{span}\{E_{p,q}~|~p\prec q\}.$$ To begin, note that if $L=\mathfrak{g}(\mathcal{P})$, then a basis for $L$ is given by $$\mathscr{B}=\{E_{p,q}~|~p\prec q\}\cup \{E_{p,p}~|~p\in\mathcal{P}\}.$$ On the other hand, if $L=\mathfrak{g}_A(\mathcal{P})$, then a basis for $L$ is given by 
$$\mathscr{B}=\{E_{p,q}~|~p\prec q\}\cup \{E_{1,1}-E_{p,p}~|~p\in\mathcal{P},1\neq p\}.$$ In either case, since $$[E_{p_1,q_1},E_{p_2,q_2}]=\delta_{q_1,p_2}E_{p_1,q_2}-\delta_{p_1,q_2}E_{p_2,q_1}\in [L,L],$$ $$[E_{p_1,p_1},E_{p_2,q_2}]=\delta_{p_1,p_2}E_{p_1,q_2}-\delta_{p_1,q_2}E_{p_2,p_1}\in [L,L],$$ and $$[E_{p_1,p_1}-E_{q_1,q_1},E_{p_2,q_2}]=\delta_{p_1,p_2}E_{p_1,q_2}-\delta_{p_1,q_2}E_{p_2,p_1}-\delta_{q_1,p_2}E_{q_1,q_2}+\delta_{q_1,q_2}E_{p_2,q_1}\in [L,L],$$ 
for $p_1,q_1,p_2,q_2\in\mathcal{P}$, and $[L,L]$ is spanned by $\{[b_1,b_2]~|~b_1,b_2\in \mathscr{B}\},$ it follows that $$\text{span}\{E_{p,q}~|~p\prec q\}\supseteq [L,L].$$ Now, if $p,q\in\mathcal{P}$ satisfy $p\prec q$, then $\frac{1}{2}(E_{p,p}-E_{q,q}), E_{p,q}\in L$ and $$\left[\frac{1}{2}(E_{p,p}-E_{q,q}), E_{p,q}\right]=E_{p,q}\in [L,L].$$ Thus, $$\text{span}\{E_{p,q}~|~p\prec q\}\subseteq [L,L]$$ and the claim is established. As $\dim(\text{span}\{E_{p,q}~|~p\prec q\})=|Rel(\mathcal{P})|$, the result follows.
\end{proof}

As a consequence of Proposition~\ref{prop:derivlp}, we obtain the following combinatorial translation of Proposition~\ref{prop:ub}(\textit{i}) for (type-A) Lie poset algebras.

\begin{proposition}\label{prop:ublpa1}
If $\mathcal{P}$ is a poset and $L=\mathfrak{g}(\mathcal{P})$ or $\mathfrak{g}_A(\mathcal{P})$, then $b(L)\le |Rel(\mathcal{P})|$.
\end{proposition}

In fact, the bound of Proposition~\ref{prop:ublpa1} is exact.

\begin{theorem}\label{thm:first}
If $\mathcal{P}$ is a poset and $L=\mathfrak{g}(\mathcal{P})$ or $\mathfrak{g}_A(\mathcal{P})$, then $$b(L)=|Rel(\mathcal{P})|.$$
\end{theorem}
\begin{proof}
Considering Proposition~\ref{prop:ublpa1}, it suffices to show that 
\begin{equation}\label{eqn:thmfirst}
    |Rel(\mathcal{P})|\le b(L).
\end{equation}
To establish (\ref{eqn:thmfirst}), we construct $x\in L$ satisfying $E_{p,q}\in \text{im}(ad_x)$, for all $p,q\in\mathcal{P}$ such that $p\prec q$. Consider $$x=\sum_{1\neq i\in\mathcal{P}}i(E_{1,1}-E_{i,i})\in L.$$ Note that $$[x,E_{1,p}]=\left(p+\sum_{1\neq i\in\mathcal{P}}i\right)E_{1,p}\neq 0,$$ for all $p\neq 1$ satisfying $1\prec p$, and $$[x,E_{p,q}]=(q-p)E_{p,q}\neq 0,$$ for all $p,q\in\mathcal{P}$ satisfying $p\neq 1$ and $p\prec q$. Thus, $$\text{im}(ad_x)\supset\{E_{p,q}~|~p\prec q\}$$ so that $$|Rel(\mathcal{P})|=\dim(\text{span}\{E_{p,q}~|~p\prec q\})\le \textup{rank}(ad_x)\le b(L).$$ The result follows.
\end{proof}

\begin{Ex} If $\mathcal{P}$ is the poset of Example~\ref{ex:stargate}, i.e., $\mathcal{P}=\{1,2,3,4\}$ with $1\prec 2\prec 3,4$, then $$b(\mathfrak{g}(\mathcal{P}))=b(\mathfrak{g}_A(\mathcal{P}))=|Rel(\mathcal{P})|=5.$$

\end{Ex}

\subsection{Nilpotent Lie poset algebras}

To determine a combinatorial translation of Proposition~\ref{prop:ub} for nilpotent Lie poset algebras we make use of the following Proposition.

\begin{proposition}\label{prop:derivnil}
If $\mathcal{P}$ is a poset and $L=\mathfrak{g}^{\prec}(\mathcal{P})$, then
\begin{enumerate}
    \item[\textup(i\textup)] $\dim([L,L])=|Rel_{\overline{C}}(\mathcal{P})|$, and
    \item[\textup(ii\textup)] $\dim(L/ Z(L))= |\{p\prec q~|~p\prec q\notin Rel_E(\mathcal{P})\}|$.
\end{enumerate}
\end{proposition}
\begin{proof}
(\textit{i}) We claim that $$[L,L]=\text{span}\{E_{p,q}~|~p\prec q\in Rel_{\overline{C}}(\mathcal{P})\}.$$ If $p,q\in\mathcal{P}$ and $p\prec q$ is not a covering relation, then there exists $r\in\mathcal{P}$ such that $p\prec r\prec q$ and $$[E_{p,r},E_{r,q}]=E_{p,q}\in [L,L];$$ that is, $$\text{span}\{E_{p,q}~|~p\prec q\in Rel_{\overline{C}}(\mathcal{P})\}\subseteq [L,L].$$  Now, since $[L,L]$ is generated by the $[E_{p_1,q_1},E_{p_2,q_2}]$, for $p_1,p_2,q_1,q_2\in\mathcal{P}$ such that $p_1\prec q_1$ and $p_2\prec q_2$, and $[E_{p_1,q_1},E_{p_2,q_2}]=\delta_{q_1,p_2}E_{p_1,q_2}-\delta_{p_1,q_2}E_{p_2,q_1}$, it follows that elements $x\in [L,L]$ must be of the form $$x=\sum_{p\prec q\in Rel_{\overline{C}}(\mathcal{P})}a_{p,q}E_{p,q};$$ that is, $$[L,L]\subseteq\text{span}\{E_{p,q}~|~p\prec q\in Rel_{\overline{C}}(\mathcal{P})\}.$$ This establishes the claim. As $\dim(\text{span}\{E_{p,q}~|~p\prec q\in Rel_{\overline{C}}(\mathcal{P})\})=|Rel_{\overline{C}}(\mathcal{P})|$, the result follows.
\bigskip

(\textit{ii}) We claim that $$Z(L)=\text{span}\{E_{p,q}~|~p\prec q\in Rel_E(\mathcal{P})\}.$$ Evidently, $E_{p,q}\in Z(L)$, for $p\prec q\in Rel_E(\mathcal{P})$, so span$\{E_{p,q}~|~p\prec q\in Rel_E(\mathcal{P})\}\subseteq Z(L).$ Now, take $z=\sum_{i\prec j}z_{i,j}E_{i,j}\in Z(L)$. If $p,q\in\mathcal{P}$ satisfy $p\prec q$ with $q$ non-maximal, then there exists $r\in\mathcal{P}$ such that $q\prec r$ and $$[z,E_{q,r}]=\sum_{k\prec q}z_{k,q}E_{k,r}-\sum_{r\prec k}z_{r,k}E_{q,k}=0;$$ in particular, $z_{p,q}=0$. If $p,q\in\mathcal{P}$ satisfy $p\prec q$ with $p$ non-minimal, then there exists $r\in\mathcal{P}$ such that $r\prec p$ and $$[z,E_{r,p}]=\sum_{k\prec r}z_{k,r}E_{k,p}-\sum_{p\prec k}z_{p,k}E_{r,k}=0;$$ in particular, $z_{p,q}=0$. Therefore, $z_{p,q}=0$, for all $p,q\in \mathcal{P}$ such that $p\prec q$ and either $p$ is non-minimal or $q$ is non-maximal. This establishes the claim, and so we have $$L/ Z(L)=\text{span}\{E_{p,q}~|~p\prec q\}/\text{span}\{E_{p,q}~|~p\prec q\in Rel_E(\mathcal{P})\}$$ $$=\text{span}\{E_{p,q}~|~p\prec q\notin Rel_E(\mathcal{P})\}.$$ As $\dim(\text{span}\{E_{p,q}~|~p\prec q\notin Rel_E(\mathcal{P})\})=|\{E_{p,q}~|~p\prec q\notin Rel_E(\mathcal{P})\}|$, the result follows.
\end{proof}

As a consequence of Proposition~\ref{prop:derivnil}, we obtain the following combinatorial translation of Proposition~\ref{prop:ub} for nilpotent Lie poset algebras.

\begin{proposition}\label{prop:ubnil}
If $\mathcal{P}$ is a poset and $L=\mathfrak{g}^{\prec}(\mathcal{P})$, then
\begin{enumerate}
    \item[\textup(i\textup)] $b(L)\le |Rel_{\overline{C}}(\mathcal{P})|$, and
    \item[\textup(ii\textup)] $b(L)\le |\{p\prec q~|~p\prec q\notin Rel_E(\mathcal{P})\}|-1$.
\end{enumerate}
\end{proposition}

Using Proposition~\ref{prop:ubnil}, the following Theorem establishes exact breadth values for nilpotent Lie poset algebras corresponding to the posets $C_n,$ $\mathbf{2\times n},$ and $T_m(n).$ In each case, the breadth of the respective algebra is equal to the number of non-covering relations in the associated poset $\mathcal{P}$.

\begin{theorem}\label{thm:breadthncov}
\begin{enumerate}
    \item[\textup{(a)}] If $\mathcal{P}=C_n$ and $L=\mathfrak{g}^{\prec}(\mathcal{P})$, then $$b(L)=|Rel_{\overline{C}}(\mathcal{P})|=\frac{(n-1)(n-2)}{2}.$$
    \item[\textup{(b)}] If $\mathcal{P}=\mathbf{2\times n}$ and $L=\mathfrak{g}^{\prec}(\mathcal{P})$, then $$b(L)=|Rel_{\overline{C}}(\mathcal{P})|=\frac{(n-1)(3n-4)}{2}.$$
    \item[\textup{(c)}] If $\mathcal{P}=\mathcal{P}=T_m(n)$ and $L=\mathfrak{g}^{\prec}(\mathcal{P})$, then $$b(L)=|Rel_{\overline{C}}(\mathcal{P})|=\frac{(n-2)m^{n+1}+(1-n)m^n+m^2}{(m-1)^2}.$$
\end{enumerate}
\end{theorem}
\begin{proof}
We will prove part (a). The proofs of parts (b) and (c) are relegated to the Appendix.

First, we show that $b(L)=|Rel_{\overline{C}}(\mathcal{P})|$. Considering Proposition~\ref{prop:ubnil}(i), it suffices to show that 
\begin{equation}\label{eqn:second}
    |Rel_{\overline{C}}(\mathcal{P})|\le b(L).
\end{equation}
To establish (\ref{eqn:second}), we construct $x\in L$ satisfying $E_{p,q}\in \text{im}(ad_x)$, for all $p,q\in\mathcal{P}$ such that $p\prec q\in Rel_{\overline{C}}(\mathcal{P})$. Consider $$x=\sum_{i=1}^{n-1}E_{i,i+1}.$$ If $1\prec p$ is not a covering relation, then $1\prec p-1\prec p$ and $$[x,-E_{1,p-1}]=E_{1,p}\in\text{im}(ad_x).$$ If $1 \neq p\prec q$ is not a covering relation, then $p\prec q-1\prec q$ and $$\left[x, -\sum_{i=0}^{p-1}E_{p-i,q-1-i}\right]=\sum_{i=0}^{p-1}E_{p-i,q-i}-\sum_{j=0}^{p-2}E_{p-j-1,q-1-j}=\sum_{i=0}^{p-1}E_{p-i,q-i}-\sum_{j=1}^{p-1}E_{p-j,q-j}=E_{p,q}\in \text{im}(ad_x).$$ Thus, $$\text{im}(ad_x)\supset\{E_{p,q}~|~p\prec q\in Rel_{\overline{C}}(\mathcal{P})\}$$ so that $$|Rel_{\overline{C}}(\mathcal{P})|=\dim(\text{span}\{E_{p,q}~|~p\prec q\in Rel_{\overline{C}}(\mathcal{P})\})\le \textup{rank}(ad_x)\le b(L).$$ It follows that $b(L)=|Rel_{\overline{C}}(\mathcal{P})|$.

Now, we show that $$|Rel_{\overline{C}}(\mathcal{P})|=\frac{(n-1)(n-2)}{2}.$$ For $p\in\{1,\hdots,n-1\}$ there are $n-p-1$ elements $q\in\mathcal{P}$ such that $p\prec q\in Rel_{\overline{C}}(\mathcal{P})$. Thus, $$|Rel_{\overline{C}}(\mathcal{P})|=\sum_{p=1}^{n-1}n-p-1=\sum_{i=1}^{n-2}i=\frac{(n-1)(n-2)}{2}.$$
\end{proof}

\begin{remark}
It appears that Theorem~\ref{thm:breadthncov} \textup{(}b\textup{)} holds more generally; that is, if $\mathcal{P}=\mathbf{m\times n}$ with $m<4$ or $n<4$ and $L=\mathfrak{g}^{\prec}(\mathcal{P})$, then $b(L)=|Rel_{\overline{C}}(\mathcal{P})|$. We conjecture that this is so \textup(see Conjecture~\ref{conj:1}\textup).
\end{remark}

It is important to note that the breadth of $\mathfrak{g}^{\prec}(\mathcal{P})$ is not always given by $|Rel_{\overline{C}}(\mathcal{P})|$. For example, the smallest poset $\mathcal{P}$ one finds satisfying $b(\mathfrak{g}^{\prec}(\mathcal{P}))< |Rel_{\overline{C}}(\mathcal{P})|$ is $\mathcal{P}=\mathcal{P}(2,1,2)$. In what follows, we show that for the family of posets $\mathcal{P}(r_0,r_1,r_2)$ the breadth of the corresponding nilpotent Lie poset algebra can be given by either of the upper bounds established in Proposition~\ref{prop:ubnil} but can also be strictly less than both.

\begin{remark}
If $\mathcal{P}=\mathcal{P}(1,n,1)$, then $\mathfrak{g}^{\prec}(\mathcal{P})$ is a generalized Heisenberg Lie algebra.
\end{remark}

Recall from Definition~\ref{def:comppos} that $\mathcal{P}(r_0,r_1,r_2)=\{b_1,\hdots,b_{r_0},m_1,\hdots,m_{r_1},t_1,\hdots,t_{r_2}\}$ with $$b_1,\hdots,b_{r_0}\prec m_1,\hdots,m_{r_1}\prec t_1,\hdots, t_{r_2}.$$ Set $\mathfrak{g}(r_0,r_1,r_2)=\mathfrak{g}^{\prec}(\mathcal{P}(r_0,r_1,r_2))$. 

\begin{theorem}\label{thm:pnrme}
If $L=\mathfrak{g}(r_0,r_1,r_2)$ with $r_1\ge r_0$ or $r_1\ge r_2$, then $b(L)=r_0r_2$.
\end{theorem}
\begin{proof}
Let $\mathcal{P}=\mathcal{P}(r_0,r_1,r_2)$. We assume that $r_1\ge r_0$, the other case following via a symmetric argument. Note that $$Rel_{\overline{C}}(\mathcal{P})=Rel_E(\mathcal{P})=\{b_i\prec t_j~|~1\le i\le r_0,1\le j\le r_2\}$$ from which it follows $$\dim([L,L])=|Rel_{\overline{C}}(\mathcal{P})|=r_0r_2.$$ Thus, by Proposition~\ref{prop:ubnil}(i), to establish the result it suffices to construct $x\in L$ such that $E_{p,q}\in \text{im}(ad_x)$, for all $p\prec q\in Rel_E(\mathcal{P})$. Consider $x=\sum_{i=1}^{r_0}E_{b_i,m_i}$. Since $$[x, E_{m_i, t_j}]=E_{b_i,t_j}\in\text{im}(ad_x),$$ for $1\le i\le r_0$ and $1\le j\le r_2$, it follows that the given $x\in L$ has the desired properties. Therefore, $$\text{im}(ad_x)\supset\{E_{p,q}~|~p\prec q\in Rel_E(\mathcal{P})\}$$ so that $$r_0r_2=|Rel_{\overline{C}}(\mathcal{P})|=\dim(\text{span}\{E_{p,q}~|~p\prec q\in Rel_E(\mathcal{P})\})\le \textup{rank}(ad_x)\le b(L).$$ The result follows.
\end{proof}

To determine a formula for $b(\mathfrak{g}(r_0,r_1,r_2))$ when $r_1<r_0,r_2$, we study the matrix of $ad_x$, denoted $M_x$, for a general element $x\in\mathfrak{g}(r_0,r_1,r_2)$. First, fix the basis $\mathscr{B}(r_0,r_1,r_2)$ of $\mathfrak{g}(r_0,r_1,r_2)$ given by $$\{E_{b_i,m_j}~|~1\le i\le r_0,~1\le j\le r_1\}\cup \{E_{m_i,t_j}~|~1\le i\le r_1,~1\le j\le r_2\}\cup \{E_{b_i,t_j}~|~1\le i\le r_0,~1\le j\le r_2\}.$$ Then every element of  $\mathfrak{g}(r_0,r_1,r_2)$ can be written as $$x=\sum_{g=1}^{r_0}\sum_{h=1}^{r_1}a_{b_g,m_h}E_{b_g,m_h}+\sum_{i=1}^{r_1}\sum_{j=1}^{r_2}a_{m_i,t_j}E_{m_i,t_j}+\sum_{k=1}^{r_0}\sum_{l=1}^{r_2}a_{b_k,t_l}E_{b_k,t_l}.$$ To study $M_x$, it will be helpful to partition $\mathscr{B}(r_0,r_1,r_2)$ into three ordered subsets:
\begin{itemize}
    \item $B_1=\{E_{b_1,m_1},E_{b_2,m_1},\hdots,E_{b_{r_0},m_1},\hdots, E_{b_1,m_{r_1}},E_{b_2,m_{r_1}},\hdots,E_{b_{r_0},m_{r_1}}\}$
    \item $B_2=\{E_{m_1,t_1},E_{m_2,t_1},\hdots,E_{m_{r_1},t_1},\hdots, E_{m_1,t_{r_2}},E_{m_2,t_{r_2}},\hdots,E_{m_{r_1},t_{r_2}}\}$ 
    \item $B_3=\{E_{b_1,t_1},E_{b_2,t_1},\hdots,E_{b_{r_0},t_1},\hdots, E_{b_1,t_{r_2}},E_{b_2,t_{r_2}},\hdots,E_{b_{r_0},t_{r_2}}\}.$
\end{itemize}
Ordering the columns of $M_x$ as $B_1,B_2,B_3$ and the rows as $B_3,B_2,B_1$, the matrix has the following form:

\begin{figure}[H]
$$\kbordermatrix{
    &  &  &  &   \\
    & -a_{m_1,t_1}I_{r_0} & -a_{m_2,t_1}I_{r_0} & \hdots & -a_{m_{r_1},t_1}I_{r_0} & A & 0  & \hdots & 0 & 0   \\
    & -a_{m_1,t_2}I_{r_0} & -a_{m_2,t_2}I_{r_0} & \hdots & -a_{m_{r_1},t_2}I_{r_0} & 0 & A  & \hdots & 0 & 0  \\
    & \vdots & \vdots & \vdots & \vdots & \vdots & \vdots & \ddots & \vdots & \vdots  \\
    & -a_{m_1,t_{r_2}}I_{r_0} & -a_{m_2,t_{r_2}}I_{r_0} & \hdots & -a_{m_{r_1},t_{r_2}}I_{r_0} & 0 & 0 &  \hdots & A & 0  \\
    & 0 & 0 & \hdots & 0 & 0 & 0 &  \hdots & 0 & 0  \\
  },$$
\caption{$M_x$}
\end{figure}

\noindent
where $I_{r_0}$ is the $r_0\times r_0$ identity matrix and

\begin{figure}[H]
$$A=\kbordermatrix{
    &  &  &  &   \\
    & a_{b_1,m_1} & a_{b_1,m_2} & \hdots & a_{b_1,m_{r_1}}  \\
    & a_{b_2,m_1} & a_{b_2,m_2} & \hdots & a_{b_2,m_{r_1}}  \\
    & \vdots & \vdots & \vdots & \vdots   \\
    & a_{b_{r_0},m_1} & a_{b_{r_0},m_2} & \hdots & a_{b_{r_0},m_{r_1}}  \\
  }.$$
\caption{Matrix form of $A$}
\end{figure}

Focusing on the section of the matrix $M_x$ corresponding to columns $B_1$, illustrated below, 

\begin{figure}[H]
$$\kbordermatrix{
    &  &  &  &   \\
    & -a_{m_1,t_1}I_{r_0} & -a_{m_2,t_1}I_{r_0} & \hdots & -a_{m_{r_1},t_1}I_{r_0}   \\
    & -a_{m_1,t_2}I_{r_0} & -a_{m_2,t_2}I_{r_0} & \hdots & -a_{m_{r_1},t_2}I_{r_0}  \\
    & \vdots & \vdots & \vdots & \vdots  \\
    & -a_{m_1,t_{r_2}}I_{r_0} & -a_{m_2,t_{r_2}}I_{r_0} & \hdots & -a_{m_{r_1},t_{r_2}}I_{r_0}   \\
    & 0 & 0 & \hdots & 0   \\
  }$$
\caption{Restriction of $M_x$ to columns $B_1$}
\end{figure}

\noindent
it is clear that the matrix $M_x$ can be row reduced in such a way that there are $sr_0$ rows with a nonzero entry in a unique column in $B_1$, for $0\le s\le r_1$, and the remaining rows contain zeros in the columns of $B_1$. Further, such a row reduction can be accomplished by performing block row operations, where the row labeled by $E_{b_i,t_j}$ is multiplied by a constant $c$ and added to the row labeled by $E_{b_i,t_k}$, for some $1\le j\neq k\le r_2$ and for all $1\le i\le r_0$.

Having performed the described block row operations, the remaining $(r_2-s)$ blocks of nonzero rows of $M_x$ labeled by elements of the form $E_{b_i,t_j}$, for $1\le i\le r_0$ and fixed $j$, with entries of 0 in the columns of $B_1$ must be of the form 

\begin{figure}[H]
$$\kbordermatrix{
    &   &  &  &  &  &  &  &  \\
    & 0  & \hdots & 0 & c_1 A & c_2A & \hdots & c_{r_2} A & 0  \\
  },$$
\caption{Remaining nonzero rows of $M_x$}
\end{figure}

\noindent
where $c_l\in \textup{\textbf{k}}$, for $1\le l\le r_2$. Such collections of rows can consist of at most $\text{rank}(A)\le r_1$ linearly independent rows. Therefore, we have that the rank of $M_x$ is bounded above by $$sr_0+r_1(r_2-s)=s(r_0-r_1)+r_1r_2,$$ for $0\le s\le r_1$; that is, we are led to the following Theorem.

\begin{theorem}\label{thm:pnrmub}
If $L=\mathfrak{g}(r_0,r_1,r_2)$ with $r_1<r_0,r_2$, then $$b(L)\le r_1(r_0-r_1)+r_1r_2=r_1(r_0+r_2-r_1).$$
\end{theorem}

\noindent
In fact, the bound of Theorem~\ref{thm:pnrmub} is exact.

\begin{theorem}\label{thm:pnrmc}
If $L=\mathfrak{g}(r_0,r_1,r_2)$ with $r_1<r_0,r_2$, then $$b(L)=r_1(r_0+r_2-r_1).$$
\end{theorem}
\begin{proof}
Considering Theorem~\ref{thm:pnrmub}, it suffices to construct an element $x\in L$ for which $\text{rank}(ad_x)\ge r_1(r_0+r_2-r_1)$. Consider $$x=\sum_{i=1}^{r_1}(E_{b_i,m_i}-E_{m_i,t_i}).$$ Since $$[x, E_{m_i, t_j}]=E_{b_i, t_j}\in \text{im}(ad_x),$$ for $1\le i\le r_1$ and $1\le j\le r_2$, and $$[x, E_{b_i, m_j}]=E_{b_i, t_j}\in \text{im}(ad_x),$$ for $1\le i\le r_0$ and $1\le j\le r_1$, it follows that
$$\{E_{b_i,t_j}~|~1\le i\le r_1,1\le j\le r_2\}\cup \{E_{b_i,t_j}~|~1\le i\le r_0,1\le j\le r_1\}\subseteq \text{im}(ad_x).$$ As $$|\{E_{b_i,t_j}~|~1\le i\le r_1,1\le j\le r_2\}|=r_1r_2,$$ $$|\{E_{b_i,t_j}~|~1\le i\le r_0,1\le j\le r_1\}|=r_0r_1,$$ and $$|\{E_{b_i,t_j}~|~1\le i\le r_1,1\le j\le r_2\}\cap \{E_{b_i,t_j}~|~1\le i\le r_0,1\le j\le r_1\}|=|\{E_{b_i,t_j}~|~1\le i\le r_1,1\le j\le r_1\}|=r_1^2,$$ we find that $$|\{E_{b_i,t_j}~|~1\le i\le r_1,1\le j\le r_2\}\cup \{E_{b_i,t_j}~|~1\le i\le r_0,1\le j\le r_1\}|= r_0r_1+r_1r_2-r_1^2=r_1(r_0+r_2-r_1).$$ Thus, $$r_1(r_0+r_2-r_1)\le \textup{rank}(ad_x)\le b(L).$$ The result follows.
\end{proof}

Combining Theorems~\ref{thm:pnrme} and~\ref{thm:pnrmc} we arrive at the following.

\begin{theorem}\label{six}
If $L=\mathfrak{g}(r_0,r_1,r_2)$, then $$b(L)=\begin{cases}
r_1(r_0+r_2-r_1), & r_1<r_0,r_2; \\
r_0r_2, & r_1\ge r_0\text{ or }r_1\ge r_2.
 \end{cases}$$
\end{theorem}

\begin{remark}\label{rem:notRelc}
Let $L=\mathfrak{g}(r_0,r_1,r_2)$. Note that 
\begin{itemize}
    \item if $r_1=1$, then $$b(L)=r_0+r_2-1= |\{p\prec q~|~p\prec q\notin Rel_E(\mathcal{P})\}|-1=\dim(L/ Z(L))-1$$ and 
    \begin{equation}\label{eqn:third}
        b(L)=r_0+r_2-1\le r_0r_2= |\{p\prec q~|~p\prec q\in Rel_{\overline{C}}(\mathcal{P})\}|=\dim([L,L]),
    \end{equation}
     where the inequality in \textup{(\ref{eqn:third})} is strict for $r_0,r_2>1$.
    
    \item if $1<r_1<r_0,r_2$, then $$b(L)=r_1(r_0+r_2-r_1)<r_1(r_0+r_2)-1= |\{p\prec q~|~p\prec q\notin Rel_E(\mathcal{P})\}|-1=\dim(L/ Z(L))-1$$ and 
    \begin{equation}\label{eqn:fifth}
        b(L)=r_1(r_0+r_2-r_1)< r_0r_2= |\{p\prec q~|~p\prec q\in Rel_{\overline{C}}(\mathcal{P})\}|=\dim([L,L]);
    \end{equation}
    the fact that the inequality in \textup{(\ref{eqn:fifth})} is strict follows by noting that $$r_0r_2=|\{(b_i,t_j)~|~1\le i\le r_0,~1\le j\le r_2\}|$$ and $$r_1(r_0+r_2-r_1)=|\{(b_i,t_j)~|~1\le i\le r_0,~1\le j\le r_2\}-\{(b_i,t_j)~|~r_1+1\le i\le r_0,~r_1+1\le j\le r_2\}|.$$
    
    \item if $r_1\ge r_0$ or $r_1\ge r_2$, then 
    \begin{equation}\label{eqn:sixth}
        b(L)=r_0r_2\le r_1(r_0+r_2)-1= |\{p\prec q~|~p\prec q\notin Rel_E(\mathcal{P})\}|-1=\dim(L/ Z(L))-1
    \end{equation}
     and $$b(L)= r_0r_2= |\{p\prec q~|~p\prec q\in Rel_{\overline{C}}(\mathcal{P})\}|=\dim([L,L]),$$ where the inequality in \textup{(\ref{eqn:sixth})} is strict when either $r_0>1$, $r_2>1$, or $r_1>r_0,r_2$.
\end{itemize}
\end{remark}

Considering Proposition~\ref{thm:breadthncov} and Remark~\ref{rem:notRelc}, it would be interesting to characterize those posets $\mathcal{P}$ for which $L=\mathfrak{g}^{\prec}(\mathcal{P})$ satisfies $b(L)=\dim([L,L])=|Rel_{\overline{C}}(\mathcal{P})|$. The following Theorem provides an obstruction to a poset $\mathcal{P}$ having the aforementioned property.

\begin{theorem}\label{thm:subpo}
Given a poset $\mathcal{P}$, let $\mathcal{Q}$ denote an induced subposet of $\mathcal{P}$ such that $[x,y]_{\mathcal{P}}\subset\mathcal{Q}$, for all $x,y\in\mathcal{Q}$ satisfying $x\prec_{\mathcal{Q}} y$. If $$b(\mathfrak{g}^{\prec}(\mathcal{Q}))<\dim([\mathfrak{g}^{\prec}(\mathcal{Q}),\mathfrak{g}^{\prec}(\mathcal{Q})]),$$ then $$b(\mathfrak{g}^{\prec}(\mathcal{P}))<\dim([\mathfrak{g}^{\prec}(\mathcal{P}),\mathfrak{g}^{\prec}(\mathcal{P})]).$$
\end{theorem}
\begin{proof}
Since for a poset $\mathcal{P}'$ one has $$[\mathfrak{g}^{\prec}(\mathcal{P}'),\mathfrak{g}^{\prec}(\mathcal{P}')]=\text{span}\{E_{p,q}~|~p\prec q\in Rel_{\overline{C}}(\mathcal{P}')\},$$ it follows that $$b(\mathfrak{g}^{\prec}(\mathcal{P}'))<\dim([\mathfrak{g}^{\prec}(\mathcal{P}'),\mathfrak{g}^{\prec}(\mathcal{P}')])$$ if and only if for all $y\in\mathfrak{g}^{\prec}(\mathcal{P}')$ there exists $p^{y}\prec q^{y}\in Rel_{\overline{C}}(\mathcal{P}')$ such that $E_{p^{y},q^{y}}\notin \text{im}(ad_{y})$. Thus, for any $y\in \mathfrak{g}^{\prec}(\mathcal{Q})$ there exists $q^{y}_1\prec q^{y}_2\in Rel_{\overline{C}}(\mathcal{Q})$ such that $E_{q^{y}_1,q^{y}_2}\notin \text{im}(ad_{y})$. 

Given $x\in\mathfrak{g}^{\prec}(\mathcal{P})$, let $x_{\mathcal{Q}}$ denote its restriction to $\mathfrak{g}^{\prec}(\mathcal{Q})$; that is, expressing $x$ in terms of the basis elements $E_{i,j}\in\mathfrak{g}^{\prec}(\mathcal{P})$, for $i\prec_{\mathcal{P}}j$, we form $x_{\mathcal{Q}}$ by removing all terms involving basis elements $E_{i,j}$ for $i\in\mathcal{P}\backslash\mathcal{Q}$ or $j\in\mathcal{P}\backslash\mathcal{Q}$. Take an arbitrary $x\in \mathfrak{g}^{\prec}(\mathcal{P})$. We claim that $E_{q^{x_{\mathcal{Q}}}_1,q^{x_{\mathcal{Q}}}_2}\notin \text{im}(ad_{x})$. Assume, toward contradiction, that $E_{q^{x_{\mathcal{Q}}}_1,q^{x_{\mathcal{Q}}}_2}\in \text{im}(ad_{x})$. Then there exists $p^*\in \mathfrak{g}^{\prec}(\mathcal{P})$ such that $[x,p^*]=E_{q^{x_{\mathcal{Q}}}_1,q^{x_{\mathcal{Q}}}_2}$. Note that for any $p\in\mathfrak{g}^{\prec}(\mathcal{P})$ we have $[x,p]=S_1+S_2$, where $$S_1=\sum_{i\prec j\in Rel_{\overline{C}}(\mathcal{Q})} a^p_{i,j}E_{i,j}\quad\quad\text{and}\quad\quad S_2=\sum_{\substack{i\prec j\in Rel_{\overline{C}}(\mathcal{P})\\i\in\mathcal{P}\backslash\mathcal{Q}\text{ or }j\in\mathcal{P}\backslash\mathcal{Q}}} a^p_{i,j}E_{i,j}.$$ Further, it must be the case that $[x_{\mathcal{Q}},p_{\mathcal{Q}}]=S_1$. If not, then there would exist $i,j\in\mathcal{Q}$ and $k\in\mathcal{P}\backslash\mathcal{Q}$ such that $i\prec k\prec j$, contradicting our assumption that $[i,j]_{\mathcal{P}}\subseteq\mathcal{Q}$. Therefore, $[x_{\mathcal{Q}},p^*_{\mathcal{Q}}]=E_{q^{x_{\mathcal{Q}}}_1,q^{x_{\mathcal{Q}}}_2}\in \text{im}(ad_{x_{\mathcal{Q}}})$, but this contradicts that $$b(\mathfrak{g}^{\prec}(\mathcal{Q})<\dim([\mathfrak{g}^{\prec}(\mathcal{Q}),\mathfrak{g}^{\prec}(\mathcal{Q})]).$$ The result follows.
\end{proof}

\begin{remark}
One can show that if $\mathcal{P}=\mathbf{4}\times\mathbf{4}$ and $L=\mathfrak{g}^{\prec}(\mathcal{P})$, then $b(L)<|Rel_{\overline{C}}(\mathcal{P})|$. Considering Theorem~\ref{thm:subpo}, it follows that if $\mathcal{P}=\mathbf{n}\times\mathbf{m}$ with $n,m>3$, then $b(L)<|Rel_{\overline{C}}(\mathcal{P})|$.
\end{remark}









\section{Epilogue} 

In this article, we focused on determining combinatorial methods for the computation of the breadth of Lie poset algebras and nilpotent Lie poset algebras. For Lie poset algebras, we found that in general breadth is given by the number of relations in the underlying poset; algebraically this value corresponds to the dimension of the associated algebra's derived algebra. In the case of nilpotent Lie poset algebras, we found that for some special families of posets the breadth is also given by the dimension of the algebra's derived algebra. However, we also found families of posets for which the breadth of the associated nilpotent Lie poset algebra is strictly less than the dimension of its derived algebra. Considering the above findings, the following question seems worth pursuing.
\bigskip

\noindent
\textbf{Question:} Does there exist a combinatorial characterization of those posets $\mathcal{P}$ for which  $$(\ast)~~~\mathfrak{g}^{\prec}(\mathcal{P})=\dim([\mathfrak{g}^{\prec}(\mathcal{P}),\mathfrak{g}^{\prec}(\mathcal{P})])=Rel_{\overline{C}}(\mathcal{P})~?$$ 
\bigskip

Theorem~\ref{thm:subpo} provides an obstruction to posets having Property $(\ast)$ and can be used to show that many well-known families posets do not have this property. For example,
\begin{itemize}
    \item $\mathbf{n}\times\mathbf{m}$, for $n,m>3$,
    \item the positive root poset of type $A_n$, for $n>6$,
    \item the positive root poset of type $B_n$ or $C_n$, for $n>3$, and
    \item the positive root poset of type $D_n$, for $n>4$,
\end{itemize}
On the other hand, data suggests the following conjecture.

\begin{conj}\label{conj:1}
If $\mathcal{P}=\mathbf{n}\times\mathbf{3}$ or the Boolean lattice $B_n$, for $n\ge 1$, and $L=\mathfrak{g}^{\prec}(\mathcal{P})$, then $b(L)=\dim([L,L])$.
\end{conj}

\noindent
Other than focusing on families of posets, one could also consider how Property $(\ast)$ behaves under various poset operations. Unfortunately, this is also seemingly wild. For example, recall that the \textit{Cartesian product of two posets $(\mathcal{P},\preceq_{\mathcal{P}})$ and $(\mathcal{Q},\preceq_{\mathcal{Q}})$} is the poset $(\mathcal{P}\times\mathcal{Q},\preceq_{\mathcal{P}\times\mathcal{Q}})$ on the set $\{(p,q)~|~p\in\mathcal{P},~q\in\mathcal{Q}\}$ such that $(p,q)\preceq_{\mathcal{P}\times\mathcal{Q}} (p',q')$ if $p \preceq_{\mathcal{P}} p'$ and $q \preceq_{\mathcal{Q}} q'$. Interestingly, one finds that
\begin{itemize}
    \item the cartesian product of two posets with Property $(\ast)$ can have Property $(\ast)$. For example, taking the cartesian product of the 3-chain with itself.
    \item the cartesian product of two posets with Property $(\ast)$ can not have Property $(\ast)$. For example, taking the cartesian product of the 4-chain with itself.
    \item using Theorem~\ref{thm:subpo}, if one of $\mathcal{P}$ or $\mathcal{Q}$ does not have Property $(\ast)$, then $\mathcal{P}\times\mathcal{Q}$ cannot have Property $(\ast)$.
\end{itemize}

Given two posets $\mathcal{P}$ and $\mathcal{Q}$, one encounters similar outcomes to that of the Cartesian product with respect to Property $(\ast)$ when forming the ordinal sum $\mathcal{P}\oplus\mathcal{Q}$, ordinal product $\mathcal{P}\otimes\mathcal{Q}$, and the poset $\widehat{\mathcal{P}}$ constructed from $\mathcal{P}$ by adjoining a new minimal and maximal element (see \textbf{\cite{Stanley}} for the definitions of these operations). It would seem that the only poset operations which behave nicely with respect to Property $(\ast)$ are forming the dual poset and the disjoint sum of a collection of posets. Recall that given posets $\mathcal{P}$ and $\mathcal{Q}$
\begin{itemize}
    \item the \textit{dual} of $\mathcal{P}$, denoted $\mathcal{P}^*$, is the poset on $\mathcal{P}$ where $i\preceq_{\mathcal{P}^*}j$ if and only if $j\preceq_{\mathcal{P}}i$, for all $i,j\in\mathcal{P}$.
    \item the \textit{disjoint sum} of $\mathcal{P}$ and $\mathcal{Q}$ is the poset $\mathcal{P}+\mathcal{Q}$ on the disjoint sum of $\mathcal{P}$ and $\mathcal{Q}$, where $s\preceq_{\mathcal{P}+\mathcal{Q}} t$ if either 
\begin{enumerate}
    \item[(i)] $s,t\in \mathcal{P}$ and $s\preceq_{\mathcal{P}} t$, or
    \item[(ii)] $s,t\in \mathcal{Q}$ and $s\preceq_{\mathcal{Q}} t$.
\end{enumerate}
\end{itemize}
For these operations it is straightforward to verify that
\begin{itemize}
    \item $\mathcal{P}$ has Property $(\ast)$ if and only if $\mathcal{P}^*$ has Property $(\ast)$, and
    \item $\mathcal{P}+\mathcal{Q}$ has Property $(\ast)$ if and only if both $\mathcal{P}$ and $\mathcal{Q}$ have Property $(\ast)$.
\end{itemize}
The above observations seem to suggest that a combinatorial characterization of Property $(\ast)$, if existent, would be very interesting.
\section{Appendix}

In this appendix, we prove parts (b) and (c) of Theorem~\ref{thm:breadthncov}.

\begin{customthm}{\ref{thm:breadthncov}}
\begin{enumerate}
    \item[\textup{(a)}] If $\mathcal{P}=C_n$ and $L=\mathfrak{g}^{\prec}(\mathcal{P})$, then $$b(L)=|Rel_{\overline{C}}(\mathcal{P})|=\frac{(n-1)(n-2)}{2}.$$
    \item[\textup{(b)}] If $\mathcal{P}=\mathbf{2\times n}$ and $L=\mathfrak{g}^{\prec}(\mathcal{P})$, then $$b(L)=|Rel_{\overline{C}}(\mathcal{P})|=\frac{(n-1)(3n-4)}{2}.$$
    \item[\textup{(c)}] If $\mathcal{P}=\mathcal{P}=T_m(n)$ and $L=\mathfrak{g}^{\prec}(\mathcal{P})$, then $$b(L)=|Rel_{\overline{C}}(\mathcal{P})|=\frac{(n-2)m^{n+1}+(1-n)m^n+m^2}{(m-1)^2}.$$
\end{enumerate}
\end{customthm}

We break the proofs of parts (b) and (c) into two lemmas.

\begin{lemma}\label{lem1}
For each of the following posets $\mathcal{P}$, if $L=\mathfrak{g}^{\prec}(\mathcal{P})$, then $b(L)=|Rel_{\overline{C}}(\mathcal{P})|$.
\begin{enumerate}
    \item[\textup{(a)}] $\mathcal{P}=\mathbf{2\times n}$
    \item[\textup{(b)}] $\mathcal{P}=T_m(n)$.
\end{enumerate}
\end{lemma}
\begin{proof}
In both cases, considering Proposition~\ref{prop:ubnil}(i), to establish the result it suffices to show that $$|Rel_{\overline{C}}(\mathcal{P})|\le b(L).$$
\bigskip

(a) For $\mathcal{P}=\mathbf{2\times n}$, to show that $|Rel_{\overline{C}}(\mathcal{P})|\le b(L)$ we construct $x\in L$ such that $E_{p,q}\in \text{im}(ad_x)$, for all $p\prec q\in Rel_{\overline{C}}(\mathcal{P})$. Consider $$x=\sum_{i=1}^{n-1}E_{i_1,(i+1)_1}+E_{i_2,(i+1)_2}.$$

\noindent
There are 4 groups of elements to consider.
\\*
 
\noindent
\textbf{Group 1}: $E_{1_1, p_k}$, for $1_1\prec p_k$, for $k=1,2$. If $1_1\prec p_k\in Rel_{\overline{C}}(\mathcal{P})$, then $1_1\prec (p-1)_k\prec p_k$ and $$[x,-E_{1_1,(p-1)_k}]=E_{1_1,p_k}\in \text{im}(ad_x).$$

\noindent
\textbf{Group 2}: $E_{1_2, p_2}$, for $1_2\prec p_2$. If $1_2\prec p_2\in Rel_{\overline{C}}(\mathcal{P})$, then $1_2\prec (p-1)_2\prec p_2$ and $$[x,-E_{1_2,(p-1)_2}]=E_{1_2,p_2}\in \text{im}(ad_x).$$

\noindent
\textbf{Group 3}: $E_{p_1, q_k}$, for $p_1\prec q_k$ where $k=1,2$ and $p\neq1$. If $p_1\prec q_k\in Rel_{\overline{C}}(\mathcal{P})$, then $p_1\prec (q-1)_k\prec q_k$ and $$\left[x,-\sum_{i=0}^{p-1}E_{(p-i)_1,(q-1-i)_k}\right]=\sum_{i=0}^{p-1}E_{(p-i)_1,(q-i)_k}-\sum_{i=0}^{p-2}E_{(p-1-i)_1,(q-1-i)_k}=E_{p_1,q_k}\in \text{im}(ad_x).$$

\noindent
\textbf{Group 4}: $E_{p_2, q_2}$, for $p_2\prec q_2$ where $p\neq1$. If $p_2\prec q_2\in Rel_{\overline{C}}(\mathcal{P})$, then $p_2\prec (q-1)_2\prec q_2$ and $$\left[x,-\sum_{i=0}^{p-1}E_{(p-i)_2,(q-1-i)_2}\right]=\sum_{i=0}^{p-1}E_{(p-i)_2,(q-i)_2}-\sum_{i=0}^{p-2}E_{(p-1-i)_2,(q-1-i)_2}=E_{p_2,q_2}\in \text{im}(ad_x).$$
\bigskip

\noindent
Thus, $$\text{im}(ad_x)\supset\{E_{p,q}~|~p\prec q\in Rel_{\overline{C}}(\mathcal{P})\}$$ so that $$|Rel_{\overline{C}}(\mathcal{P})|=\dim(\text{span}\{E_{p,q}~|~p\prec q\in Rel_{\overline{C}}(\mathcal{P})\})\le \textup{rank}(ad_x)\le b(L).$$ It follows that $b(L)=|Rel_{\overline{C}}(\mathcal{P})|$.
\bigskip

(b) For $\mathcal{P}=T_m(n)$, evidently the result holds for $n=1,2$. For $n>2$, to show that $|Rel_{\overline{C}}(\mathcal{P})|\le b(L)$ we construct $x\in L$ such that $E_{p,q}\in \text{im}(ad_x)$, for all $p\prec q\in Rel_{\overline{C}}(\mathcal{P})$. Consider $$x=\sum_{k=1}^{n-1}\sum_{i=1}^{m^{k-1}}\sum_{j=0}^{m-1} E_{i_k,(mi-j)_{k+1}}.$$ Note that $x=\sum E_{p,q}$, where the sum is over all covering relations $p\prec q$ in $\mathcal{P}$. Also note that $$\{p\prec q~|~p\prec q\in Rel_{\overline{C}}(\mathcal{P})\}=\{i_k\prec j_l~|~l-k>1\text{ and }m^{l-k}(i-1)+1\le j\le m^{l-k}i\}.$$ We will show that $E_{p,q}\in \text{im}(ad_x)$, for all $p\prec q\in Rel_{\overline{C}}(\mathcal{P})$, in $n-2$ steps.
\bigskip

\noindent
\textbf{Step 0}: $i_k\prec j_n$. If $i_k\prec j_n\in Rel_{\overline{C}}(\mathcal{P})$, then there exists $0\le l\le m-1$ such that $i_k\prec (mi-l)_{k+1}\prec j_n$ and $$[x,E_{(mi-l)_{k+1}, j_n}]=E_{i_k,j_n}\in \text{im}(ad_x).$$

\noindent
\textbf{Step d}: $i_k\prec j_{n-d}$. If $i_k\prec j_{n-d}\in Rel_{\overline{C}}(\mathcal{P})$, then there exists $0\le l\le m-1$ such that $i_k\prec (mi-l)_{k+1}\prec j_{n-d}$ and $$[x,E_{(mi-l)_{k+1}, j_{n-d}}]=E_{i_k,j_{n-d}}-\sum_{t=0}^{m-1}E_{(mi-l)_{k+1},(mj-t)_{n-d+1}}\in \text{im}(ad_x).$$ As a consequence of \textbf{Step $\mathbf{d-1}$} we may conclude that $\sum_{t=0}^{m-1}E_{(mi-l)_{k+1},(mj-t)_{n-d+1}}\in \text{im}(ad_x)$. So, $E_{i_k,j_{n-d}}\in\text{im}(ad_x)$.
\bigskip

\noindent
Thus, $$\text{im}(ad_x)\supset\{E_{p,q}~|~p\prec q\in Rel_{\overline{C}}(\mathcal{P})\}$$ so that $$|Rel_{\overline{C}}(\mathcal{P})|=\dim(\text{span}\{E_{p,q}~|~p\prec q\in Rel_{\overline{C}}(\mathcal{P})\})\le \textup{rank}(ad_x)\le b(L).$$ It follows that $b(L)=|Rel_{\overline{C}}(\mathcal{P})|$.
\end{proof}

\begin{lemma}\label{lem2}
\begin{enumerate}
    \item[\textup{(a)}] If $\mathcal{P}=\mathbf{2\times n}$, then $$|Rel_{\overline{C}}(\mathcal{P})|=\frac{(n-1)(3n-4)}{2}.$$
    \item[\textup{(b)}] If $\mathcal{P}=T_m(n)$, then $$|Rel_{\overline{C}}(\mathcal{P})|=\frac{(n-2)m^{n+1}+(1-n)m^n+m^2}{(m-1)^2}.$$
\end{enumerate}
\end{lemma}
\begin{proof}
(a) By induction. Let $\mathcal{P}_n=\mathbf{2\times n}$ and $$f(n)=|Rel_{\overline{C}}(\mathcal{P}_n)|.$$ For $n=1$, direct computation shows that $f(1)=0$.  Assume the result holds for $1\le n-1$. Note that one can form $\mathcal{P}_n$ from $\mathcal{P}_{n-1}$ by adjoining a new maximal element $n_2$ satisfying $(n-1)_2\prec n_2$ as well as a new element $n_1$ satisfying $(n-1)_1\prec n_1\prec n_2$. Thus, $$f(n)=f(n-1)+|Rel_{\overline{C}}(\mathcal{P}_n,n_2)|+|Rel_{\overline{C}}(\mathcal{P}_n,n_1)|,$$ where $$|Rel_{\overline{C}}(\mathcal{P}_n,n_2)|=|\{j_1\prec n_2~|~1\le j\le n-1\}\cup \{j_2\prec n_2~|~1\le j\le n-2\}|=2n-3$$ and $$|Rel_{\overline{C}}(\mathcal{P}_n,n_1)|=|\{j_1\prec n_2~|~1\le j\le n-2\}|=n-2;$$ that is, $$f(n)=\frac{(n-2)(3n-7)}{2}+2n-3+n-2=\frac{(n-1)(3n-4)}{2}.$$ The result follows.
\bigskip

(b) Let $\mathcal{P}_n=T_m(n)$, $$f_1(n)=|Rel_{\overline{C}}(\mathcal{P}_n)|,$$ and $$f_2(n)=\frac{(n-2)m^{n+1}+(1-n)m^n+m^2}{(m-1)^2}.$$ By definition we have that $f_1(1)=f_1(2)=0$, $$f_2(1)=\frac{(1-2)m^{1+1}+(1-1)m^1+m^2}{(m-1)^2}=\frac{(-1)m^{2}+(0)m+m^2}{(m-1)^2}=\frac{-m^{2}+m^2}{(m-1)^2}=\frac{0}{(m-1)^2}=0,$$ and $$f_2(2)=\frac{(2-2)m^{2+1}+(1-2)m^2+m^2}{(m-1)^2}=\frac{(0)m^{3}+(-1)m^2+m^2}{(m-1)^2}=\frac{-m^{2}+m^2}{(m-1)^2}=\frac{0}{(m-1)^2}=0.$$ Thus, $f_1(n)=f_2(n)$, for $n=1,2$. For $n>2$, note that the induced poset defined by $\mathcal{P}_n\backslash\{1_1\}$ is the disjoint sum of $m$ copies of $\mathcal{P}_{n-1}$. Consequently, $$f_1(n)=mf_1(n-1)+|Rel_{\overline{C}}(\mathcal{P}_n,1_1)|,$$ where $$|Rel_{\overline{C}}(\mathcal{P}_n,1_1)|=|\{1_1\prec i_k~|~2<k\le n,1\le i\le m^{k-1}\}|=\sum_{k=2}^{n-1} m^{k}=m^2\left(\frac{m^{n-2}-1}{m-1}\right);$$ that is, $$f_1(n)=mf_1(n-1)+m^2\left(\frac{m^{n-2}-1}{m-1}\right).$$ Now, note that $$mf_2(n-1)+m^2\left(\frac{m^{n-2}-1}{m-1}\right)=m\left(\frac{(n-3)m^{n}+(2-n)m^{n-1}+m^2}{(m-1)^2}\right)+m^2\left(\frac{m^{n-2}-1}{m-1}\right)$$ $$=\frac{(n-3)m^{n+1}+(2-n)m^{n}+m^3}{(m-1)^2}+\frac{m^{n}-m^2}{m-1}$$ $$=\frac{(n-3)m^{n+1}+(2-n)m^{n}+m^3}{(m-1)^2}+\frac{m^{n+1}-m^n-m^3+m^2}{(m-1)^2}$$ $$=\frac{(n-2)m^{n+1}+(1-n)m^n+m^2}{(m-1)^2}=f_2(n).$$ Therefore, since $f_1(n)$ and $f_2(n)$ satisfy the same initial conditions and recursive relation, they are equal.
\end{proof}

Combining the results of Lemmas~\ref{lem1} and~\ref{lem2} establishes the remaining cases of Theorem~\ref{thm:breadthncov}.


\begin{thebibliography}{abcd}


\bibitem{CM}
V. Coll and N. Mayers. ``The index of Lie poset algebras." \textit{J. Combin. Theory Ser. A}, 177, 2021.

\bibitem{nil} V. Coll, N. Mayers, and N. Russoniello. ``The index of nilpotent Lie poset algebras." \textit{Linear Algebra Appl.}, 605: 118 -- 129, 2020.


\bibitem{CG}
V. Coll and M. Gerstenhaber. 
``Cohomology of Lie semidirect products and poset algebras." \textit{J. Lie Theory}, 26(1): 79--95, 2016.

\bibitem{D}
J. Dixmier. ``Enveloping Algebras." Vol. 14. Newnes, 1977.

\bibitem{Graaf}
W.A. de Graaf. ``Classification of solvable Lie algebras." \textit{Exp. Math.}, 14(1): 15--25, 2005.

\bibitem{dissertation}
B. Khuhirun. ``Classification of nilpotent Lie algebras with small breadth." Doctoral Dissertation, North Carolina State University. ProQuest Dissertations Publishing, 2014.

\bibitem{breadth1}
B. Khuhirun, K. Misra, and E. Stitzinger. ``On nilpotent Lie algebras of small breadth." \textit{J. Algebra}, 444: 328--338, 2015.

\bibitem{Green}
C. R. Leedham-Green, P.Neumann, and J. Wiegold.
``The breadth and the class of a finite p-group."
\textit{J. Lond. Math. Soc.}, 2(1): 409--420, 1969.

\bibitem{nilco}
L. Lampret and A. Vavpeti\v{c}. ``(Co)homology of Lie algebras via algebraic Morse theory." \textit{J. Algebra}, 463: 254--277, 2016.


\bibitem{Parm}
G. Parmeggiani and B. Stellmacher.``p-Groups of small breadth." \textit{J. Algebra}, 213(1): 52--68, 1999.

\bibitem{Remm}
E. Remm. ``Breadth and characteristic sequence of nilpotent Lie algebras." \textit{Comm. Algebra}, 45(7): 2956--2966, 2017.


\bibitem{Ro}
G. Rota. ``On the foundations of combinatorial theory I.  Theory of M\"obius functions." \textit{Probab. Theory Related Fields}, 2(4): 340--368, 1964.

\bibitem{Stanley}
R. Stanley. ``Enumerative Combinatorics, vol. 1." Cambridge Stud. Adv. Math., vol. 49, Cambridge University Press, Cambridge,
1997 





\end{thebibliography}
\end{document}